\documentclass{patmorin}
\usepackage{pat}
\usepackage[T1]{fontenc}
\usepackage[utf8]{inputenc}
\usepackage{mathtools}
\usepackage{amssymb}
\usepackage{thmtools}
\usepackage{thm-restate}
\newcommand{\sm}{\smallsetminus}

\usepackage[inline]{enumitem}
\usepackage{soul}
\usepackage[normalem]{ulem}

\crefname{p}{}{}
\creflabelformat{p}{#2(#1)#3}

\usepackage{todonotes}

\usepackage[longnamesfirst,numbers,sort&compress]{natbib}

\newcommand{\Aut}{\mathsf{Aut}}
\newcommand{\Autpm}{\mathsf{Aut}^{\pm}}
\newcommand{\Sym}{\mathsf{Sym}}
\newcommand{\Dih}{\mathsf{Dih}}
\newcommand{\SD}{\mathsf{SD}}

\title{\MakeUppercase{Automorphism Groups in Extremal Families of Polyhedral Graphs}%
  \thanks{This research was partly funded by NSERC.\\This research was partly funded by Programme for Young Researchers `Rita Levi Montalcini' PGR21DPCWZ.}}

\author{
Riccardo W. Maffucci
\thanks{Dipartimento di Matematica, Universit\`a di Torino
Via Carlo Alberto 10, Turin 10123, Italy}
\and\qquad
Bobby Miraftab%
\thanks{School of Computer Science, Carleton University, Ottawa, Canada.}
}

\date{}

\begin{document}
\maketitle

\begin{abstract}
We study automorphism groups in five extremal families of polyhedral graphs.
For every $n\ge14$, we prove that every minimum-order $3$-polytopal graph
containing a vertex of each degree $3,4,\ldots,n$ is asymmetric.  The proof
uses an exact planar defect decomposition, a complete description of the
high-degree tail, and a saturation theorem for the subgraph induced by the
uniquely high-degree vertices.  Duality gives the corresponding asymmetry
result for minimum-face polyhedra containing faces of every size
$3,4,\ldots,n$.  For the three polyhedral graphs whose complements are also
polyhedral, we determine the ordinary and extended automorphism groups and
identify the extended group
\[
\Autpm(G_{13})\cong (C_2\times C_2)\rtimes C_4.
\]
Next, we classify automorphism groups of radius-one polyhedra. In the
unique-dominating-vertex case they are cyclic or dihedral, and in the
triangulated case the possibilities are
\[
1,\qquad C_2,\qquad C_3,\qquad C_2\times C_2,\qquad S_3.
\]
For polyhedra that are unigraphic among the class of self-dual, we show that their automorphism group is either $1$ or $C_2$.
\\
Finally, we consider polyhedra that are products of graphs, for each of the four standard graph products, and we classify them according to their automorphism group.
\end{abstract}
\noindent
{\bf Keywords:} Automorphism group, Planar graph, Polyhedron, Degree sequence, Dominating vertex, Self-dual, Graph product.
\\
{\bf MSC(2020):} 05C25, 05C10, 	05C35, 05C76, 52B05, 52B10.

\section{Introduction}

A polyhedral graph is a finite simple $3$-connected planar graph. These are the wireframes of polyhedral solids \cite{radste}, have a unique immersion in the sphere \cite{whit32}, and their duals are also polyhedra. For an account of further properties of this class of graphs, we refer the reader to \cite[Section 1.1]{maffucci2025regularity}.

Throughout, $C_m$ denotes the cyclic group of order $m$, and
$\Dih_m$ the dihedral group of order $2m$. Apart from the Platonic solids, the automorphism group of a polyhedron is one of \cite{whit32,klavik2015constructive}
\[C_m,\quad\Dih_m,\quad C_m\times C_2,\quad\Dih_m\times C_2.\]
The automorphism group of the tetrahedron is the symmetric group $S_4$, that of the cube and octahedron is $S_4\times C_2$, that of the dodecahedron and icosahedron is $A_5\times C_2$, where $A_5$ is the alternating group \cite{whit32,klavik2015constructive}.

We consider several extremal settings.  
First, we study minimum-order
polyhedral graphs containing a vertex of every degree $3,4,\ldots,n$.  An
exact decomposition of the planar degree-sum defect determines the entire
high-degree tail and forces a dense subgraph on the uniquely high-degree
vertices.  
This yields asymmetry without any additional facial hypothesis.
Second, we determine the ordinary and extended automorphism groups of the
three polyhedral graphs whose complements are polyhedral. 
Third, we classify the automorphism groups of polyhedra having a dominating vertex, with a sharper classification in the triangulated case. Fourth, we classify the automorphism groups of graphs that are unigraphic among the self-dual polyhedra, meaning that they are the only realisation of their degree sequence among the self-dual polyhedra. Fifth, we classify the polyhedral graph products (Cartesian, Kronecker, strong, lexicographic) in terms of automorphism groups.

\section{Preliminaries}

The following result can be found in
\cite{Bowen1966ValenciesPlanarGraphs,Chvatal1969PlanarityDegreeVertices}.

\begin{lem}\label{lem:BC}
Let $G$ be a planar graph of order $p$, and write its degree sequence in weakly
decreasing order as $d_1\ge d_2\ge\cdots\ge d_p$.  If
\[
3\le k\le \frac{p+4}{3},
\]
then
\[
\sum_{i=1}^{k}d_i\le 2p+6k-16.
\]
\end{lem}

\begin{lem}[{\cite[Theorem 2]{Maffucci2023PolytopalFamilies}}]
\label{lem:JGT-pn}
Let $n\ge14$.  If $G$ is a $3$-polytopal graph of minimum order among
$3$-polytopal graphs containing at least one vertex of degree $i$ for every
$3\le i\le n$, then
\[
|V(G)|=
 p(n):=
\left\lceil \frac{n^2-11n+62}{4}\right\rceil.
\]
\end{lem}

The following is known as the Whitney flag rigidity.

\begin{lem}\label{lem:flag-rigidity}
Let $G$ be a $3$-connected planar graph.  An automorphism of $G$ is determined
by the image of one flag, that is, by the image of one incident triple
$(v,e,F)$ with $v\in e\subseteq\partial F$.  Consequently, an automorphism
that fixes a facial triangle pointwise is the identity.
\end{lem}

\begin{proof}
By Whitney's uniqueness theorem, every automorphism of $G$ acts on its unique
spherical embedding, up to reflection.  The flag graph of a connected
cellularly embedded graph is connected, so an automorphism of the embedded
graph is determined by the image of one flag.  If a facial triangle is fixed
pointwise, then one of its vertex--edge--face flags is fixed, and the
automorphism is trivial.
\end{proof}

\begin{lem}[{\cite[Theorem 2]{Maffucci2022Complements}}]
\label{lem:comp-class}
There are exactly three polyhedral graphs whose complements are also
polyhedral.  They are
\[
g_{14.8.12},\qquad g_{14.8.13},\qquad g_{14.8.39}.
\]
All three are self-complementary and have degree sequence $4,4,4,4,3,3,3,3$.
\end{lem}

\section{Degree-minimal polyhedra}

For a vertex set $X\subseteq V(G)$, write $e(X):=|E(G[X])|$ and let $e(X,\overline X)$ denote the number of edges with one endpoint in
$X$ and the other in $V(G)\sm X$.  
Let $B_X$ be the spanning bipartite subgraph of $G$ whose edge set consists precisely of the $X$--$(V(G)\sm X)$ edges.

We now provide the planar defect decomposition. 

\begin{lem}\label{lem:planar-defect}
Let $G$ be a planar graph of order $p$, and let $X\subseteq V(G)$ have
$h:=|X|\ge3$.  Then
\[
2p+6h-16-\sum_{x\in X}d_G(x)
=
2\bigl(3h-6-e(X)\bigr)
+
\bigl(2p-4-e(X,\overline X)\bigr).
\]
Both terms on the right are nonnegative.  
Consequently, $\sum_{x\in X}d_G(x)\le 2p+6|X|-16$
for every $X$ of size at least three.
\end{lem}

\begin{proof}
The identity follows from
\[
\sum_{x\in X}d_G(x)=2e(X)+e(X,\overline X).
\]
The graph $G[X]$ is planar, so $e(X)\le3h-6$.  The graph $B_X$ is a simple
bipartite planar graph on $p$ vertices, so
$e(X,\overline X)\le2p-4$.
\end{proof}

Taking $X$ to be a set of the $k$ highest-degree vertices recovers the
numerical inequality in \Cref{lem:BC}.  The defect decomposition shows
that this numerical inequality itself is valid for every $k\ge3$, without an
upper restriction on $k$.

\begin{cor}\label{cor:small-defect}
Assume $p\ge4$, and put
\[
D_X:=2p+6|X|-16-\sum_{x\in X}d_G(x).
\]
Then the following hold.
\begin{enumerate}[label=(\roman*)]
\item If $D_X=0$, then $G[X]$ is a triangulation and $B_X$ is connected,
with every facial boundary walk of length four.
\item If $D_X=1$, then $G[X]$ is a triangulation, $B_X$ is connected, and
all facial boundary walks of $B_X$ have length four except for one of length
six.
\item If $D_X=2$, then exactly one of
\[
\bigl(3|X|-6-e(X),\ 2p-4-e(X,\overline X)\bigr)
\in\{(0,2),(1,0)\}
\]
holds.
\end{enumerate}
\end{cor}

\begin{proof}
We set $a:=3|X|-6-e(X)$, and $b:=2p-4-e(X,\overline X)$.
\Cref{lem:planar-defect} gives $D_X=2a+b$, with $a,b\ge0$.
The asserted values of $a$ and $b$ follow immediately.
It remains to justify the facial statements.  
If a component $C$ of a simple bipartite planar graph has order $q$, then
\[
|E(C)|\le
\begin{cases}
0,&q=1,\\
1,&q=2,\\
2q-4,&q\ge3.
\end{cases}
\]
If the graph is disconnected, has total order $p\ge4$, and these bounds are
summed over its components, the result is at most $2p-6$.  Hence a bipartite
planar graph with $2p-4$ or $2p-5$ edges is connected.
If $e(B_X)=2p-4$, Euler's formula gives $p-2$ faces and $2e(B_X)=4p-8=4(p-2)$.
Every facial boundary walk in a connected simple bipartite plane graph on at
least four vertices has even length at least four.  Hence every facial
boundary walk has length four.
If $e(B_X)=2p-5$, Euler's formula gives $p-3$ faces and
\[
2e(B_X)=4p-10=4(p-3)+2.
\]
Thus exactly one facial boundary walk has length six and all the others have
length four.
\end{proof}

For $n\ge14$, we define
\[
\varepsilon_n:=
\begin{cases}
0,&n\equiv1,2\pmod4,\\
1,&n\equiv0,3\pmod4.
\end{cases}
\]

\begin{thm}\label{thm:Pn-fixed-high-degree}
Let $n\ge14$, and let $G$ be a $3$-polytopal graph of minimum order among
those containing a vertex of every degree $3,4,\ldots,n$.  Choose one vertex
$v_j$ of degree $j$ for each $6\le j\le n$, and put $R:=\{v_6,v_7,\ldots,v_n\}$.
Then, for every $W\subseteq V(G)\sm R$,
\[
\sum_{w\in W}\bigl(d_G(w)-6\bigr)\le\varepsilon_n.
\]
Consequently:
\[
\begin{array}{ll}
n\equiv1,2\pmod4:
&\text{every vertex outside $R$ has degree at most $6$;}\\[1mm]
n\equiv0,3\pmod4:
&\text{at most one vertex outside $R$ has degree $7$, and every}\\
&\text{other vertex outside $R$ has degree at most $6$.}
\end{array}
\]
In particular, degree $7$ occurs exactly once when
$n\equiv1,2\pmod4$, and once or twice when $n\equiv0,3\pmod4$.
Every degree $8,9,\ldots,n$ occurs exactly once, no vertex has degree greater
than $n$, and every uniquely occurring vertex is fixed by $\Aut(G)$.
\end{thm}

\begin{proof}
We set  $p:=|V(G)|$, $k:=n-5$, and  $L:=\sum_{j=6}^{n}j$.
By \Cref{lem:JGT-pn}, direct substitution gives $2p+6k-16=L+\varepsilon_n$.
We invoke \Cref{lem:planar-defect} to $X=R\cup W$.  
Since
\[
\sum_{x\in X}d_G(x)=L+\sum_{w\in W}d_G(w)
\quad\text{and}\quad
|X|=k+|W|,
\]
we obtain
\[
L+\sum_{w\in W}d_G(w)
\le
L+\varepsilon_n+6|W|.
\]
This is the required inequality.
When $\varepsilon_n=0$, taking $W$ to be a singleton excludes every extra
vertex of degree at least seven.  When $\varepsilon_n=1$, a singleton excludes
every extra vertex of degree at least eight, while two extra vertices of
degree seven would contribute $2$ and are therefore impossible.  The
remaining assertions follow because automorphisms preserve degrees.
\end{proof}

\begin{cor}\label{cor:large-face-tail}
Let $n\ge14$, and let $P$ have the minimum possible number of faces among
polyhedral graphs containing an $i$-gonal face for every $3\le i\le n$.
Choose one $j$-gonal face $F_j$ for each $6\le j\le n$, and let
$\mathcal W$ be any collection of other faces.  Then
\[
\sum_{F\in\mathcal W}\bigl(|\partial F|-6\bigr)\le\varepsilon_n.
\]
Thus, if $n\equiv1,2\pmod4$, every other face has size at most six.  If
$n\equiv0,3\pmod4$, at most one other face is heptagonal and all remaining
faces have size at most six.  In particular, the heptagonal face is unique in
the first case and occurs once or twice in the second, while every face size
$8,9,\ldots,n$ occurs exactly once.
\end{cor}

\begin{proof}
Apply ~\Cref{thm:Pn-fixed-high-degree} to the dual graph $P^*$.
\end{proof}

\begin{lem}\label{lem:fixed-star}
Let $G$ be a $3$-connected planar graph, and let $X\subseteq V(G)$ be fixed
pointwise by $\Aut(G)$.  If $\Delta(G[X])\ge3$, then $\Aut(G)=1$.
\end{lem}

\begin{proof}
Choose $v\in X$ having three distinct neighbours
$x_1,x_2,x_3\in X$.  Every automorphism fixes $v$ and the three incident
edges $vx_1,vx_2,vx_3$.  By Whitney uniqueness, the induced action on the
cyclic order of the edges incident with $v$ is dihedral.  A nonidentity
dihedral permutation of a cycle fixes at most two positions, so the local
action at $v$ is the identity.  Two consecutive incident edges and the face
between them are therefore fixed.  Hence an incident vertex--edge--face flag
is fixed, and \Cref{lem:flag-rigidity} implies that the automorphism is
the identity.
\end{proof}

\begin{thm}\label{thm:core-saturation}
Let $n\ge14$, and let $G$ be a $3$-polytopal graph of minimum order among
those containing a vertex of every degree $3,4,\ldots,n$.  Put
$p:=|V(G)|$.
If $n\equiv1,2\pmod4$, let $X:=\{v_j:d_G(v_j)=j,\ 7\le j\le n\}$, and 
$h:=|X|=n-6$.
Then $e(X)=3h-6$, and  $e(X,\overline X)=2p-4$.
Thus $G[X]$ is a triangulation, and $B_X$ is a connected spanning bipartite
plane graph all of whose facial boundary walks have length four.
If $n\equiv0,3\pmod4$, let $X:=\{v_j:d_G(v_j)=j,\ 8\le j\le n\}$,
$h:=|X|=n-7$.
Then precisely one of the following alternatives holds:
\[
\begin{array}{c|c}
e(X)&e(X,\overline X)\\ \hline
3h-6&2p-6\\
3h-7&2p-4.
\end{array}
\]
In the first alternative, $G[X]$ is a triangulation.  In the second,
$G[X]$ is a $2$-connected plane graph whose faces are all triangles except
for one quadrilateral, and $B_X$ is a connected spanning bipartite plane
graph all of whose facial boundary walks have length four.
\end{thm}

\begin{proof}
By ~\Cref{thm:Pn-fixed-high-degree}, all vertices in the displayed set
$X$ are uniquely determined by their degrees.
Suppose first that $n\equiv1,2\pmod4$.  Then
\[
\sum_{x\in X}d_G(x)
=
\sum_{j=7}^{n}j
=
\frac{n(n+1)}2-21,
\qquad
h=n-6,
\]
and \Cref{lem:JGT-pn} gives
\[
2p=
\frac{n(n+1)}2-6n+31.
\]
Consequently,
\[
2p+6h-16-\sum_{x\in X}d_G(x)=0.
\]
Both nonnegative terms in \Cref{lem:planar-defect} therefore vanish,
and ~\Cref{cor:small-defect}(i) gives the facial conclusion for
$B_X$.
Now suppose that $n\equiv0,3\pmod4$.  Then
\[
\sum_{x\in X}d_G(x)
=
\sum_{j=8}^{n}j
=
\frac{n(n+1)}2-28,
\qquad
h=n-7,
\]
and
\[
2p=
\frac{n(n+1)}2-6n+32.
\]
Thus
\[
2p+6h-16-\sum_{x\in X}d_G(x)=2.
\]
~\Cref{cor:small-defect}(iii) gives the two alternatives in the
table.  In the second alternative,
~\Cref{cor:small-defect}(i) gives the facial conclusion for $B_X$.

It remains to describe $G[X]$ when $e(X)=3h-7$.  Here $h\ge8$.  A
disconnected planar graph on $h$ vertices has at most $3h-9$ edges.  A
connected planar graph with a cut vertex has at most $3h-8$ edges, as follows
by splitting it at a cut vertex and applying the planar edge bound to the
resulting pieces.  Hence $G[X]$ is $2$-connected.  Euler's formula gives
$2h-5$ faces, and
\[
2e(X)=6h-14=3(2h-5)+1.
\]
Every facial boundary has length at least three, so exactly one face has
length four and all the others have length three.
\end{proof}

\begin{thm}
\label{thm:universal-asymmetry}
Let $n\ge14$, and let $G$ be a $3$-polytopal graph of minimum order among
those containing a vertex of every degree $3,4,\ldots,n$.  
Then $\Aut(G)=1$.
\end{thm}

\begin{proof}
Let $X$ be the high-degree core from ~\Cref{thm:core-saturation}.  
Every vertex of $X$ is fixed by
~\Cref{thm:Pn-fixed-high-degree}.  Moreover,
\[
e(G[X])\ge3|X|-7>|X|,
\]
so the average degree of $G[X]$ is greater than two.  Hence
$\Delta(G[X])\ge3$.  \Cref{lem:fixed-star} now gives
$\Aut(G)=1$.
\end{proof}

\begin{cor}\label{cor:dual-asymmetry}
Let $n\ge14$, and let $G$ have the minimum possible number of faces among
polyhedral graphs containing an $i$-gonal face for every $3\le i\le n$.
Then $\Aut(G)=1$.
\end{cor}

\begin{proof}
The dual $G^*$ has the minimum possible number of vertices among polyhedral
graphs containing a vertex of every degree $3,4,\ldots,n$.  Apply
~\Cref{thm:universal-asymmetry} to $G^*$ and use the natural
identification $\Aut(G)\cong\Aut(G^*)$.
\end{proof}

\begin{cor}
\label{cor:top-layer-triangulation}
Let $G$ be as in ~\Cref{thm:universal-asymmetry}, let
$p:=|V(G)|$, and put $k:=n-5$.  Let $A\subseteq V(G)$ have size $k$ and
satisfy
\[
d_G(a)\ge d_G(b)
\qquad
\text{for every $a\in A$ and $b\in V(G)\sm A$}.
\]
Then $e(A)=3k-6=3n-21$.
Moreover,
\[
e(A,\overline A)=
\begin{cases}
2p-5,
 &n\equiv0,3\pmod4\text{ and degree $7$ is unique},\\
2p-4,
 &\text{otherwise}.
\end{cases}
\]
Thus $G[A]$ is a triangulation.  In the second line, $B_A$ is connected and
all its facial boundary walks have length four.  In the first line, $B_A$ is
connected, exactly one facial boundary walk has length six, and all the
others have length four.
\end{cor}

\begin{proof}
If $n\equiv1,2\pmod4$, the degree multiset on $A$ is $\{6,7,8,\ldots,n\}$.
Its sum is $2p+6k-16$, so the defect $D_A$ is zero.
Suppose that $n\equiv0,3\pmod4$.  If degree $7$ is unique, the degree multiset
on $A$ is again $\{6,7,8,\ldots,n\}$, whose sum is now one less than
$2p+6k-16$; hence $D_A=1$.  If degree $7$ occurs twice, the degree multiset on
$A$ is
\[
\{7,7,8,\ldots,n\},
\]
and $D_A=0$.
In all cases $D_A\in\{0,1\}$.  Since the induced-edge defect in
\Cref{lem:planar-defect} has coefficient two, it must vanish.  Therefore
$e(A)=3k-6$.  The stated cut sizes and facial descriptions now follow from
~\Cref{cor:small-defect}.
\end{proof}

\section{Polyhedral graphs with polyhedral complements}

For the next theorem, $uv$ denotes the unordered edge $\{u,v\}$.

\begin{thm}\label{thm:exact-complement-aut}
Let $G_{12}\cong g_{14.8.12}$, $G_{13}\cong g_{14.8.13}$, and $G_{39}\cong g_{14.8.39}$ be the three graphs of \Cref{lem:comp-class}.  They admit the following labelled edge sets:
\[
\begin{aligned}
E(G_{12})=\{&
ab,bc,cd,da,ax,dx,xy,ay,yc,yz,zb,zw,bw,cw\},\\[1mm]
E(G_{13})=\{&
ab,bc,cd,da,uv,ux,xy,vy,au,ax,bv,by,dx,cy\},\\[1mm]
E(G_{39})=\{&
ab,bc,cd,da,at,tb,ts,sx,sy,xy,ax,dx,by,cy\}.
\end{aligned}
\]
Then $\Aut(G_{12})
=
\left\langle
(a\,b)(c\,y)(d\,z)(x\,w)
\right\rangle
\cong C_2$, and  $\Aut(G_{13})
=
\left\langle
(c\,v)(d\,u),
(a\,b)(c\,d)(u\,v)(x\,y),
(a\,x)(b\,y)
\right\rangle
\cong C_2^3$,
and
$\Aut(G_{39})
=
\langle\alpha,\beta\rangle
\cong\Dih_4$,
where $\alpha=(a\,b\,y\,x)(c\,s\,d\,t)$, $\beta=(a\,b)(c\,d)(x\,y)$.
\end{thm}

\begin{proof}
The displayed permutations preserve the corresponding edge lists, so they
generate subgroups of the claimed automorphism groups.

For $G_{12}$, the degree-four vertices are $a,b,c,y$.  Among them, $a$ and
$b$ are exactly those whose two degree-three neighbours are adjacent.  Thus
$a$ has at most two possible images.  Once the image of $a$ is chosen, the
induced cycle on the degree-four vertices and the adjacency pattern to the
degree-three vertices force every other image.  Hence
$|\Aut(G_{12})|\le2$, and the listed involution gives equality.

For $G_{13}$, the degree-four vertices are $a,b,x,y$.  The image of $a$ has
at most four choices.  Once the image of the degree-three neighbour
$u\in N(a)$ is chosen, there are at most two choices, and the remaining
images are forced: $x$ is the unique common neighbour of $a$ and $u$, the
other degree-three neighbour of $a$ is then forced, and the rest follows from
the edge list.  Thus $|\Aut(G_{13})|\le8$.  The three listed commuting
involutions generate a subgroup of order eight, proving
$\Aut(G_{13})\cong C_2^3$.

For $G_{39}$, the degree-four vertices are again $a,b,x,y$.  The image of
$a$ has at most four choices, and after choosing the image of the
degree-three neighbour $t\in N(a)$ there are at most two choices.  The rest
is forced: $b$ is the unique common neighbour of $a$ and $t$, and then the
remaining vertices follow successively from the edge list.  Hence
$|\Aut(G_{39})|\le8$.  The displayed permutations satisfy
\[
\alpha^4=\beta^2=1,
\qquad
\beta\alpha\beta=\alpha^{-1},
\]
and generate a subgroup of order eight.  Therefore
$\Aut(G_{39})\cong\Dih_4$.
\end{proof}

\begin{cor}\label{cor:no-asymmetric-complement-pair}
If $G$ and $\overline G$ are both polyhedral, then $\Aut(G)\ne1$.  More
precisely,
\[
\Aut(G)\in\{C_2,C_2^3,\Dih_4\}.
\]
Thus no polyhedral graph with polyhedral complement is asymmetric.
\end{cor}

\begin{proof}
This is immediate from \Cref{lem:comp-class} and
~\Cref{thm:exact-complement-aut}.
\end{proof}

\subsection{Extended automorphism groups}

\begin{defn}
For a graph $G$ on vertex set $V$, define
\[
\Autpm(G):=
\left\{
\sigma\in\Sym(V):
\sigma(E(G))=E(G)
\text{ or }
\sigma(E(G))=\binom{V}{2}\sm E(G)
\right\}.
\]
Thus $\Aut(G)$ is an index-two subgroup of $\Autpm(G)$ whenever $G$ is
self-complementary.  Elements of $\Autpm(G)\sm\Aut(G)$ are called
complementing permutations.
\end{defn}

\begin{thm}\label{thm:autpm-three}
For the three graphs in ~\Cref{thm:exact-complement-aut}, the following
hold.

\begin{enumerate}[label=(\roman*)]
\item $\Autpm(G_{12})
=
\langle\theta_{12}\rangle
\cong C_4$, $\theta_{12}:=(a\,d\,b\,z)(c\,x\,y\,w)$.

\item We set $r:=(c\,v)(d\,u)$, $s:=(a\,x)(b\,y)$,
$t:=\theta_{13}:=(a\,c\,b\,u)(d\,x\,v\,y)$.
Then
\[
\Autpm(G_{13})
=
\langle r,s,t\rangle
\cong
(C_2\times C_2)\rtimes C_4,
\]
where the generator of $C_4$ interchanges the two standard generators of
$C_2\times C_2$.  
More explicitly,
\[
\Autpm(G_{13})
\cong
\left\langle
r,s,t:
 r^2=s^2=t^4=1,
 [r,s]=1,
 trt^{-1}=s,
 tst^{-1}=r
\right\rangle.
\]
Furthermore, $\Aut(G_{13})=\langle r,s,t^2\rangle\cong C_2^3$,
$Z\bigl(\Autpm(G_{13})\bigr)
=
\langle rs,t^2\rangle
\cong C_2^2,$
$\Autpm(G_{13})'
=
\langle rs\rangle
\cong C_2.
$
The group has one identity element, seven involutions, and eight elements of
order four.  All eight complementing permutations have order four.  The
index-two extension
\[
1\longrightarrow\Aut(G_{13})
\longrightarrow\Autpm(G_{13})
\longrightarrow C_2
\longrightarrow1
\]
is nonsplit.

\item
$\Autpm(G_{39})
=
\langle\gamma,\beta\rangle
\cong\SD_{16}$,
where $\gamma:=(a\,c\,b\,s\,y\,d\,x\,t)$,
$\beta=(a\,b)(c\,d)(x\,y)$,
and $\SD_{16}
=
\langle R,S:R^8=S^2=1,\ SRS=R^3\rangle$
is the semidihedral group of order $16$.
\end{enumerate}
\end{thm}

\begin{proof}
A direct check against the edge lists in
~\Cref{thm:exact-complement-aut} shows that
$\theta_{12}$, $t$, and $\gamma$ send edges to nonedges and nonedges to
edges.  Since all three graphs are self-complementary by
\Cref{lem:comp-class}, $|\Autpm(G_i)|=2|\Aut(G_i)|$ for 
$i\in\{12,13,39\}$.
For $G_{12}$, we have $\theta_{12}^2=(a\,b)(c\,y)(d\,z)(x\,w)$,
the nonidentity automorphism of $G_{12}$.  
Thus $\Autpm(G_{12})\cong C_4$.
For $G_{13}$, let $q:=(a\,b)(c\,d)(u\,v)(x\,y)$,
the second generator of $\Aut(G_{13})$ in
~\Cref{thm:exact-complement-aut}.  
Direct multiplication gives
\[
t^2=rq,
\qquad
trt^{-1}=s,
\qquad
tst^{-1}=r.
\]
Hence $q=rt^2$, so $\langle r,s,t\rangle$ contains
$\Aut(G_{13})$ and a complementing permutation.  It therefore equals
$\Autpm(G_{13})$ and has order $16$.

The subgroup $V:=\langle r,s\rangle$ is a normal Klein four-group,
$\langle t\rangle\cong C_4$, and the order computation gives
$V\cap\langle t\rangle=1$.  This proves the asserted semidirect-product
description and presentation.  Since $t^2$ is central,
\[
\Aut(G_{13})=V\times\langle t^2\rangle\cong C_2^3.
\]
The conjugation action interchanges $r$ and $s$, from which
\[
Z\bigl(\Autpm(G_{13})\bigr)=\langle rs,t^2\rangle
\quad\text{and}\quad
\Autpm(G_{13})'=\langle rs\rangle
\]
follow.  
The complementing coset is
$Vt\cup Vt^3$.
For $v\in V$, $(vt)^2=v(tvt^{-1})t^2$,
which is either $t^2$ or $rst^2$, and is always a nonidentity involution.
The same holds for $vt^3$.  Hence all eight complementing elements have
order four.  The remaining seven nonidentity elements of
$\Aut(G_{13})\cong C_2^3$ are involutions.  Since the complementing coset
contains no involution, the displayed index-two extension cannot split.

For $G_{39}$, the permutation $\gamma$ has order eight, and
\[
\beta^2=1,
\qquad
\beta\gamma\beta=\gamma^3.
\]
Thus $\langle\gamma,\beta\rangle$ is a quotient of $\SD_{16}$.  It contains
the eight powers of $\gamma$ and also
$\beta\notin\langle\gamma\rangle$, so it has order $16$.  Hence it equals
$\Autpm(G_{39})$ and is isomorphic to $\SD_{16}$.
\end{proof}

\begin{cor}\label{cor:complementer-inventory}
Define
\[
P_G(z):=
\sum_{\theta\in\Autpm(G)\sm\Aut(G)}
z^{\operatorname{ord}(\theta)}.
\]
For the three complement-polyhedral graphs,
\[
P_{G_{12}}(z)=2z^4,
\qquad
P_{G_{13}}(z)=8z^4,
\qquad
P_{G_{39}}(z)=4z^4+4z^8.
\]
Consequently, these inventories, unlike the sets of possible orders alone,
distinguish all three graphs.  More precisely, the complementing cycle types
are
\[
\begin{array}{c|c}
G&\text{complementing cycle types}\\ \hline
G_{12}&2\text{ permutations of type }(4,4),\\
G_{13}&8\text{ permutations of type }(4,4),\\
G_{39}&4\text{ permutations of type }(4,4)
       \text{ and }4\text{ of type }(8).
\end{array}
\]
\end{cor}

\begin{proof}
The formulas for $G_{12}$ and $G_{13}$ follow from
~\Cref{thm:autpm-three}.  For $G_{39}$,
\[
\alpha=\gamma^2
\qquad\text{and}\qquad
\Aut(G_{39})=\langle\gamma^2,\beta\rangle.
\]
The complementing coset consists of the four odd powers of $\gamma$, all of
order eight, and the four elements $\gamma^{2i+1}\beta$.  Since
$\beta\gamma\beta=\gamma^3$, $\gamma^{2i+1}\beta)^2=\gamma^4$,
so these four elements have order four.  The cycle types follow by direct
multiplication from the displayed permutations: the odd powers of $\gamma$
are $8$-cycles, while every remaining complementer is a product of two
$4$-cycles.
\end{proof}

\section{Automorphism groups of radius-one polyhedra}

The radius of a graph is one if and only if it has a dominating vertex.  A
polyhedron of radius one with at least six vertices either has exactly one
dominating vertex or is isomorphic to $P_{n-2}+K_2$
\cite[Lemma 4.1]{maffucci2025classification}.  If $u$ is the unique
dominating vertex, then $G-u$ is $2$-connected and outerplanar; equivalently,
it is outerplanar and Hamiltonian
\cite[Lemma 3.3]{maffucci2025regularity}.

\begin{thm}
\label{thm:radius-one}
Let $G$ be a polyhedron with exactly one dominating vertex.  Then
\[
\Aut(G)\cong C_m\quad(m\ge1),
\qquad\text{or}\qquad
\Aut(G)\cong\Dih_m\quad(m\ge2).
\]
Conversely, every finite cyclic group and every dihedral group $\Dih_m$ with
$m\ge2$ occurs as the automorphism group of such a polyhedron.
\end{thm}

\begin{proof}
Let $u$ be the unique dominating vertex and put $H:=G-u$.  The vertex $u$ is
fixed by every automorphism, and restriction gives $\Aut(G)\cong\Aut(H)$.
The graph $H$ is $2$-connected and outerplanar.  Its outer boundary is its
unique Hamilton cycle, so every automorphism of $H$ induces a rotation or a
reflection of that cycle.  Thus $\Aut(H)$ is a subgroup of a dihedral group.
Every subgroup of a dihedral group is cyclic or dihedral, proving the first
assertion.

We now realize all the listed groups.  If $H$ is a $2$-connected outerplanar
graph with no dominating vertex, then $K_1+H$ is planar and $3$-connected:
removing the apex and one further vertex leaves the connected graph obtained
from $H$ by deleting one vertex, while every deletion of two vertices that
leaves the apex also leaves the graph connected.  Hence $K_1+H$ is a
polyhedral graph with a unique dominating vertex, and
\[
\Aut(K_1+H)\cong\Aut(H).
\]

For $C_1$, take a $6$-cycle with the noncrossing chords $02$ and $03$.  For
$C_2$, take a $5$-cycle with the chord $02$.  Direct inspection of the unique
outer cycle gives the claimed automorphism groups.
For $m\ge3$, place
\[
a_0,b_0,c_0,a_1,b_1,c_1,\ldots,a_{m-1},b_{m-1},c_{m-1}
\]
in this cyclic order, with indices modulo $m$, and add the noncrossing chords
\[
a_i a_{i+1}
\qquad\text{and}\qquad
a_i c_i
\quad(0\le i<m).
\]
The degrees along the unique boundary cycle repeat as $5,2,3$.  Therefore the
only dihedral symmetries of the boundary preserving the graph are the $m$
rotations through whole blocks, and the automorphism group is $C_m$.

For $\Dih_2$, take a $6$-cycle with the chord $03$.  For $m\ge3$, place
\[
a_0,b_0,a_1,b_1,\ldots,a_{m-1},b_{m-1}
\]
in cyclic order and add all chords $a_i a_{i+1}$.  The vertices $a_i$ induce
an $m$-cycle when $m\ge4$ and a triangle when $m=3$; each $b_i$ is the unique
degree-two vertex adjacent to $a_i$ and $a_{i+1}$.  Hence every automorphism
is induced by a rotation or reflection of the indices, and all such
symmetries occur.  The automorphism group is $\Dih_m$.  In every construction
above, the outerplanar base has no dominating vertex.
\end{proof}

\begin{thm}
\label{thm:radius-one-triangulated}
Let $G$ be a triangulated polyhedron with exactly one dominating vertex.
Then $\Aut(G)\in
\{1,C_2,C_3,C_2\times C_2,S_3\}$.
All five groups occur.
\end{thm}

\begin{proof}
Let $u$ be the unique dominating vertex and put $H:=G-u$.  Since $G$ is a
triangulation, $H$ is maximal outerplanar, equivalently a triangulation of a
polygon.  Realize the boundary polygon as a regular polygon.  The rotation
subgroup of $\Aut(H)$ acts by geometric rotations preserving the polygonal
triangulation.

The centre of the polygon lies either in the interior of a triangular cell
or in the relative interior of a diagonal.  The cell containing the centre
is invariant under the rotation subgroup.  In the first case the rotation
subgroup acts faithfully on the three vertices of the triangle, so its order
is at most three.  In the second case it acts faithfully on the two endpoints
of the diagonal, so its order is at most two.  Therefore the rotation
subgroup has order at most three.  Combining this with
~\Cref{thm:radius-one} leaves precisely
\[
C_1,
\quad C_2,
\quad C_3,
\quad\Dih_2,
\quad\Dih_3,
\]
where $\Dih_2\cong C_2\times C_2$ and $\Dih_3\cong S_3$.

It remains to realize the five possibilities.  In the following table, the
vertices of $H$ are numbered cyclically from $0$ to $|V(H)|-1$, and only the
diagonals are listed:
\[
\begin{array}{c|c|l}
\Aut(H)&|V(H)|&\text{diagonals}\\ \hline
1&7&02,26,36,46\\
C_2&6&14,15,24\\
C_3&9&14,17,18,24,47,57\\
C_2\times C_2&8&02,04,06,24,46\\
S_3&6&13,15,35.
\end{array}
\]
Each diagonal set is noncrossing and has $|V(H)|-3$ edges, so each $H$ is
maximal outerplanar.  Since automorphisms preserve the unique boundary cycle,
the cyclic boundary-degree sequences
\[
\begin{array}{c|c}
1&(3,2,4,3,3,2,5),\\
C_2&(2,4,3,2,4,3),\\
C_3&(2,5,3,2,5,3,2,5,3),\\
C_2\times C_2&(5,2,4,2,5,2,4,2),\\
S_3&(2,4,2,4,2,4)
\end{array}
\]
restrict the possible dihedral actions to the groups displayed in the first
column; the listed diagonals are invariant under every one of those actions.
None of these bases has a dominating vertex.  Adjoining the unique apex
therefore gives the required triangulated polyhedra.
\end{proof}

\begin{cor}\label{cor:path-join}
For every $n\ge6$, $\Aut(P_{n-2}+K_2)\cong C_2\times C_2$.
\end{cor}

\begin{proof}
The two vertices belonging to $K_2$ are precisely the dominating vertices,
so their set is invariant and they may be fixed or interchanged.  The
remaining vertices induce $P_{n-2}$, which may be fixed or reversed.  These
two commuting involutions are independent and generate
$C_2\times C_2$.
\end{proof}

\begin{cor}
\label{cor:all-radius-one}
Let $G$ be a radius-one polyhedron on at least six vertices.  Then $\Aut(G)$
is cyclic or dihedral.  If $G$ is a triangulation, then
\[
\Aut(G)\in
\{1,C_2,C_3,C_2\times C_2,S_3\}.
\]
\end{cor}

\begin{proof}
If $G$ has a unique dominating vertex, apply
Theorems~\ref{thm:radius-one} and \ref{thm:radius-one-triangulated}.  Otherwise,
the cited radius-one classification gives $G\cong P_{n-2}+K_2$, and
~\Cref{cor:path-join} gives
$\Aut(G)\cong C_2\times C_2\cong\Dih_2$.
\end{proof}

\section{Automorphism groups of polyhedra that are unigraphic among the self-dual}
Apart from pyramids, the only self-dual polyhedra that are the unique self-dual realisation of their degree sequences are the graphs $S(m,n)$, $m\geq n\geq 4$ from \cite[Figure 2]{maffucci2026self}.

\begin{thm}
One has
\[\Aut(S(m,n))=\begin{cases}1&m>n\geq 4\\C_2&m=n\geq 4.\end{cases}\]
\end{thm}
\begin{proof}
For $m>n\geq 4$, there is a unique vertex of degree $m$, a unique vertex of degree $n$, a unique face of length $m$, and a unique face of length $n$. Moreover, the vertex of degree $m$ lies on the $n$-gonal face and vice versa. These two vertices and two faces are thereby fixed by every automorphism of $S(m,n)$. Finally, of the two neighbours of the vertex of degree $m$, one lies on both the $n$-gon and $m$-gon, and the other on the $n$-gon but not on the $m$-gon. Thereby every automorphism of $S(m,n)$ fixes a flag and is thus the identity.

For $m=n\geq 4$, we add to the considerations above the possibility of an automorphism swapping the two vertices of maximal degree. Hence the automorphism group is indeed $C_2$.
\end{proof}

\section{Automorphism groups of polyhedra that are products of graphs}
Among all products of graphs that one may define, the ones with the most relevant properties are the Cartesian, Kronecker (i.e., direct, tensor), strong, and lexicographic \cite{hammack2016handbook}.

There exist two classes of polyhedral Cartesian products: the stacked prisms (i.e., products of a cycle and a path), and the prisms of outerplanar, Hamiltonian graphs \cite{maffucci2024classification}. Only the tetrahedron $K_2\boxtimes K_2$, and the graph $K_3\boxtimes K_3$ are polyhedral strong products \cite{maffucci2024classification}, and only the tetrahedron and octahedron are polyhedral lexicographic products \cite{maffucci2025regularity}.

\begin{thm}
One has $\Aut(K_3\boxtimes K_3)=\Dih_4$.
\\
Let $G$ be a polyhedral Cartesian product. Then we have $\Aut(CY_n\square P_\ell)=\Dih_n\times C_2$, for all $\ell\geq 2$
and
\[\Aut(H\square K_2)\in\{C_m\times C_2\ (m\geq 1),\ \Dih_m\times C_2\ (m\geq 2)\},\]
and all these possibilities occur.
\end{thm}
\begin{proof}
In $K_3\boxtimes K_3$ there is a unique vertex of degree $8$ that is thus fixed by every automorphism. Once its neighbours of degree $4$ are mapped, the automorphism is completely determined. Hence $K_3\boxtimes K_3$ has the same automorphism group as the square.

For $\Aut(CY_n\square P_\ell)$ we simply apply $\Aut(CY_n)=\Dih_n$, and the order of vertices along the paths of the stacked prism is either fixed or reversed.

Similarly, for $G=H\square K_2$, we have $\Aut(H)\in\{C_m,\ \Dih_m\}$ by \Cref{thm:radius-one}.
\end{proof}

The polyhedral Kronecker products were classified in \cite{maffucci2024classification}, and further investigated in \cite{de2024cancellation}.

\begin{thm}
Let $G$ be a polyhedral Kronecker product. Then
\[\Aut(G)\in
\{C_{m'}\times C_2,\Dih_{m'}\times C_2\},\qquad m'\geq 1,\]
and each of these possibilities occurs.
\\
Moreover if $G=J\otimes K_2$ with $J$ planar, then $\Aut(G)\in
\{C_2\times C_2,C_2\times C_2\times C_2\}$ and both of these possibilities occur.
\end{thm}
\begin{proof}
Referring to \cite[Theorem 1.7]{maffucci2024classification}, one has $G=J\otimes K_2$, where $J$ is uniquely determined \cite[Theorem 1.10]{de2024cancellation}. Thereby,
\[\Aut(G)=\Aut(J)\times C_2.\]

Now $J$ is constructed as in \cite[Algorithm 1.8]{maffucci2024classification} starting from a plane, bipartite graph $J'$ that is either $3$-connected, or semi-hyper-$2$-connected. In either case, we may find in $J$ a unique cycle $Y$ of smallest length (boundary of the region $r$ of $J'$ defined in \cite[Theorem 1.7]{maffucci2024classification}) such that the graph $Y_{\text{out}}$ spanned by the vertices on and outside of $Y$ is bipartite. Moreover, $J$ is obtained from $J'$ by adding distinct edges $a_1b_1$, $a_2b_2$, \dots, $a_mb_m$, $m\geq 2$, where the endpoints of these edges lie on $Y$. Hence one has $\Aut(J)\leq\Aut(J')$.

If $J'$ is $3$-connected, then an element of $\Aut(J')$ is completely determined once the vertices of the face bounded by $Y$ have been mapped. If $J'$ is semi-hyper-$2$-connected, then according to \cite[Theorem 1.7]{maffucci2024classification} every $2$-cut of $J'$ lies on $Y$. Hence in this case also, an element of $\Aut(J')$ is completely determined once the vertices of the face bounded by $Y$ have been mapped. Therefore, $\Aut(J)\leq\Aut(J')\leq\Aut(Y)$.

If $J$ is planar, %then $Y_{\text{int}}$ is outerplanar and Hamiltonian, since according to \cite[Theorem 1.7]{maffucci2024classification}, it consists of the cycle $Y$ where we have added (at least two) non-crossing diagonals $a_1b_1$, $a_2b_2$, \dots, $a_mb_m$. By Theorem \ref{thm:radius-one}, $\Aut(Y_{\text{int}})\in
%\{C_{m'},\Dih_{m'}\}$, for some $m'$ dividing the length of $Y$. 
then we know rather more: since by planarity of $J$ the vertices
\[a_1,a_2,\dots,a_m,b_m,b_{m-1},\dots,b_1\]
appear around $Y$ is this order, then the edges $a_1b_1$ and $a_mb_m$ are either fixed or swapped by every element of $\Aut(Y)$, so that $\Aut(Y)\in
\{C_2,C_2\times C_2\}$. Either of these possibilities may occur. Indeed, take
\[Y=u,a_1,a_2,\dots,a_m,v,b_m,b_{m-1},\dots,b_1.\]
To obtain $\Aut(J)=C_2$, we form $J$ by adding the edges of the path $a_1,b_1,a_2,b_2,\dots,a_m,b_m$ to $J'$. To obtain $\Aut(J)=C_2\times C_2$, one forms $J$ by adding to $J'$ the pairwise disjoint edges $a_1b_1$, $a_2b_2$, \dots, $a_mb_m$.

If instead $J$ is non-planar, then according to \cite[Theorem 1.7]{maffucci2024classification}, the vertices
\[a_1,a_2,\dots,a_m,b_1,b_2,\dots,b_m\]
lie on $Y$ in this order. 
One has $\Aut(Y)\in\{C_{m'},\Dih_{m'}\}$,
where $m'$ divides the length of $Y$. If
$Y=a_1,a_2,u_1,\dots,b_{m-1},b_m,u_{m}$ and $J'$ is the $3m$-gonal prism, with one base given by $Y$, then $\Aut(J)=\Dih_{m}$. To the other base of this prism we now add an $m$-cycle consisting of one out of three vertices (starting at $a_1$ and finishing at $b_{m-1}$, say). Then $\Aut(J)=C_{m}$.
\end{proof}

\section{Open questions}
We close the paper with some questions: ideally, one would like a criterion to establish, for every polyhedron, the corresponding automorphism group.

\begin{op}
Fixing the number of vertices $p$, what is the proportion of polyhedra with trivial automorphism group, as $p$ tends to infinity?
\end{op}

\begin{op}
Can we characterise and construct the class of polyhedra with automorphism group $C_\ell$, where $\ell$ is a fixed prime number?
\end{op}

Let $G$ be a two-ended, locally finite, $3$-connected planar
quasi-transitive graph.  It is known that $G$ admits a finite-width
decomposition over a double ray and that $\Aut(G)$ is virtually cyclic;
see \cite{MiraftabRuehmann2022,DunwoodyPlanarCovers}.
Canonical decompositions of locally finite planar quasi-transitive graphs
are also available; see \cite{GiocantiPlanarQuasiTransitive}.

\begin{op}
Which virtually cyclic groups occur as the full automorphism group of a two-ended, locally finite, $3$-connected planar quasi-transitive graph?

More precisely, determine the possible finite normal subgroups
$F\lhd\Aut(G)$ and the possible extensions having quotient
$\mathbb Z$ or $\Dih_\infty$.  What additional restrictions are imposed
by planarity and $3$-connectivity?
\end{op}

\begin{op}
Let $\Omega(G)=\{\omega^-,\omega^+\}$
and let $\Aut_0(G):=
\{\alpha\in\Aut(G):
  \alpha(\omega^-)=\omega^-
  \text{ and }\alpha(\omega^+)=\omega^+\}$.
Suppose that the action of $\Aut(G)$ on $\Omega(G)$ is nontrivial.  When
does the exact sequence
\[
1\longrightarrow\Aut_0(G)
\longrightarrow\Aut(G)
\longrightarrow C_2
\longrightarrow1
\]
split?  Equivalently, must there exist an involution interchanging the two
ends?  If not, which nonsplit extensions can occur for infinite polyhedral
graphs?
\end{op}

\begin{op}
Define the end-connectivity of a two-ended graph by
\[
\kappa_\infty(G):=
\min\bigl\{|S|:
 S\subseteq V(G)\text{ is finite and the two ends of $G$ lie in
 distinct components of }G-S\bigr\}.
\]
For a $3$-connected graph one has $\kappa_\infty(G)\ge3$.
How does $\kappa_\infty(G)$ restrict the finite normal subgroup and the
extension structure of $\Aut(G)$?

In particular, determine the possible automorphism groups when $\kappa_\infty(G)=3$.
What stronger conclusions hold when $G$ is a plane triangulation?
\end{op}

It is worth mentioning that the case $\kappa_{\infty}(G)=2$ has been studies for planar Cayley graphs here \cite{MiraftabStavropoulos2021}.

\bibliographystyle{plainurlnat}
\bibliography{ref}

\end{document}